\documentclass[11pt]{amsart}
\usepackage[utf8]{inputenc}
\usepackage[margin=2.7cm]{geometry}
\usepackage{graphicx,xcolor}
\usepackage{amscd,amsmath,amsfonts,amssymb,amsthm}

\newcommand{\N}{{\mathbb N}}
\newcommand{\R}{{\mathbb R}}
\newcommand{\W}{W_0^{1,2}(\Omega)}

\newtheorem{theorem}{Theorem}
\newtheorem{lemma}[theorem]{Lemma}

\newtheorem{remark}[theorem]{Remark}

\begin{document}
	
	\hfill\today\bigskip
	
	\title[]{Positive solutions for a Kirchhoff problem of Brezis-Nirenberg type in dimension four}
	
	\author{Giovanni Anello}
	\address[G. Anello]{Department of Mathematical and Computer Sciences, Physical Sciences and Earth Sciences\\
		University of Messina\\
		Viale F. Stagno d'Alcontres, 31 - 98166 Messina, Italy}
	\email{\tt ganello@unime.it}
	
	\author{Luca Vilasi$^\dagger$}
	\address[L. Vilasi]{Department of Mathematical and Computer Sciences, Physical Sciences and Earth Sciences\\
		University of Messina\\
		Viale F. Stagno d'Alcontres, 31 - 98166 Messina, Italy}
	\email{\tt lvilasi@unime.it}
	
	\keywords{Kirchhoff problem; Brezis-Nirenberg problem; dimension four; existence; positive solution\\
		\phantom{aa} 2020 AMS Subject Classification: 35J20, 35J25, 35J61\\
		\phantom{aa} $^\dagger$Corresponding author: lvilasi@unime.it.
	}

\begin{abstract}
We consider a Kirchhoff problem of Brezis-Nirenberg type in a smooth bounded domain of $\R^4$ with Dirichlet boundary conditions. Our approach, novel in this framework and based upon approximation arguments, allows us to cope with the interaction between the higher order Kirchhoff term and the critical nonlinearity, typical of the dimension four. We derive several existence results of positive solutions, complementing and improving earlier results in the literature. In particular, we provide explicit bounds of the parameters $b$ and $\lambda$ coupled, respectively, with the higher order Kirchhoff term and the subcritical nonlinearity, for which the existence of solutions occurs.
\end{abstract}

\maketitle

\section{Introduction}
Let $\Omega\subset\R^4$ be a smooth bounded domain, let $a,b,\lambda>0$, and $p\in(2,4)$. In this paper we consider the following semilinear Brezis-Nirenberg type problem,
\begin{equation}\tag{$P_\lambda$}\label{problem}
	\left\{\begin{array}{ll}
		\displaystyle{-\left( a+b\int_{\Omega}|\nabla u|^2\, dx\right)\Delta u= u^3+\lambda u^{p-1}} & \text{ in } \Omega \\
		u> 0 & \text{ in } \Omega \\
		u=0 & \text{ on } \partial\Omega.
	\end{array}\right.
\end{equation}
Problems with this structure, in which the leading operator is coupled with a non-local coefficient, date back to the work of Kirchhoff \cite{kir1883vorlesungen}, where the following equation
\begin{equation}\label{kirchhoff}
	\varrho \frac{\partial^2 u}{\partial t^2} - \left( \frac{P_0}{h}+\frac{E}{2L}\int_0^L \left| \frac{\partial u}{\partial x}\right| ^2 dx\right) \frac{\partial^2 u}{\partial x^2} = 0
\end{equation}
was proposed as a nonlinear extension of d'Alambert's wave equation for free vibrations of elastic strings. The constants in \eqref{kirchhoff} have the following meaning: $u=u(x,t)$ is the transverse string displacement at the space coordinate $x$ and time $t$, $L$ is the length of the string, $h$ is the area of the cross-section, $E$ is the Young modulus of the material, $\varrho$ is the mass density and $P_0$ is the initial tension. Problem \eqref{problem}, in this sense, represents a stationary (and multidimensional) version of \eqref{kirchhoff} and includes an external force term. The purely local case $a=1$ and $b=0$ was the subject of the groundbreaking paper \cite{bn} by Brezis and Nirenberg, which opened the way for the variational treating of critical problems with subcritical perturbations, and inspired a huge amount of subsequent papers.

Coming back to the non-local setting, over the last decades
Kirchhoff-type problems have been a very active field of investigation, both in the subcritical and in the critical regime. The latter, as well-known, is characterized by the lack of compactness of Sobolev embeddings which makes the variational approach harder. The existence and multiplicity of solutions of Kirchhoff critical problems have been recently addressed by different tools: truncation methods, Nehari manifold approach, concentration-compactness principles, just to name a few (cf. \cite{alvcorfig2010on,camvil2014a,fan2015multiple,farsil2021on,f,mamor2024asymptotic,nai2014positive,naishi2020existence} and the references enclosed).

In the very recent \cite{mamor2024asymptotic}, the asymptotic profile of positive ground states of a critical Kirchhoff equation with power-type nonlinearities was investigated via rescaling arguments in the dimensions $N=3$ and $N=4$. The authors showed in addition that this analysis is related to a mass constrained problem, for which they established the existence, non-existence and (sharp) asymptotic behavior of positive normalized solutions.

In \cite{farsil2021on} the authors considered a  Brezis-Nirenberg type problem with a general subcritical term in higher dimensions. In the case of a pure power (the one of \eqref{problem}), by carrying out a detailed fibering map analysis, they found out the existence of two critical hyperbolas dividing the plane $(a,b)$ into regions where the energy functional has different topological properties. In particular, they proved that, for suitable constants $C_2(N)>C_1(N)>0$ depending only on $N$, and for a certain $\lambda_0^*=\lambda_0^*(a,b)$, the energy has: one global minimizer for $a^\frac{N-4}{2}b> C_1(N)$ and for $\lambda>\lambda_0^*$; one global minimizer for $a^\frac{N-4}{2}b= C_1(N)$ and $\lambda>0$; one local minimizer for $a^\frac{N-4}{2}b>C_1(N)$ and $\lambda$ in a left neighborhood of $\lambda_0^*$; a critical point of mountain pass type when $a^\frac{N-4}{2}b> C_2(N)$ and $\lambda$ again in a left neighborhood of $\lambda_0^*$, or when $a^\frac{N-4}{2}b=C_2(N)$ and $\lambda$ is large (cf. \cite[Theorems 1.1, 1.2, 1.3]{farsil2021on}, as well as Theorems 1.5, 1.6, 1.7). The techniques used in this paper are peculiar of higher dimensions ($N>4$) and do not apply to the "limiting" case $N=4$.

The dimension four, indeed, displays one more difficulty due to the appearance of a new phenomenon: the critical Sobolev exponent $2^*=2N/(N-2)$ in this case is four and this produces an interaction between the higher order Kirchhoff term $\left( \int_\Omega |\nabla u|^2 dx\right)^2 $ and the critical nonlinearity $\int_\Omega u^4 dx$. In this direction we mention the recent contributions \cite{f} and \cite{naishi2020existence}, which motivated our study.

The subject of \cite{f} is a critical problem set in dimension $N\geq 3$ with a general Kirchhoff function $M$ and a subcritical term $f$, which encode, as particular cases, those of problem \eqref{problem}, i.e., $M(t)=a+bt$ and $f(t)=t^{2^*-1}$. By using truncation arguments, the author proved the existence of one positive mountain pass solution for any $\lambda>0$ sufficiently large. The solutions are moreover shown to bifurcate from the zero one as $\lambda\to +\infty$. Similar results had been earlier obtained in \cite{alvcorma2005positive}, still by truncation techniques, but using completely different estimates.

In \cite{naishi2020existence}, instead, the authors dealt with  a problem with the same structure as \eqref{problem} (with $a=1$) via the Nehari manifold method and the concentration-compactness analysis for Palais-Smale sequences. Their work improves earlier results obtained in
\cite{nai2014the} by removing some technical assumptions on the parameter $\lambda$ and the domain $\Omega$. The paper addresses separately the linear case $p=2$ and the nonlinear one $p\in(2,4)$. Concerning the latter, the main existence result is the following one:

\begin{theorem}[Theorem 1.4 of \cite{naishi2020existence}]\label{thmnaishi}
	Let $p\in(2,4)$ and let
	\begin{equation}\label{defS-2}
		S:=\inf\left\lbrace \int_\Omega |\nabla u|^2 dx: u\in \W, \int_\Omega u^4 dx=1 \right\rbrace.
	\end{equation}
	Then,
	\begin{itemize}
		\item[$(i)$] there exists $b_*=b_*(p)\in(0,S^{-2})$ such that \eqref{problem} has at least one solution for all $b\in(0,b_*)$ and for all $\lambda>0$;
		\item[$(ii)$] for all $b\in(0,S^{-2})$ there exists $\lambda_*>0$ such that \eqref{problem} has at least one solution for every $\lambda\in(0,\lambda_*)$.
	\end{itemize}
\end{theorem}

As already said, the conclusion is achieved by constructing Palais-Smale sequences on suitable submanifolds of the Nehari manifold which are shown to be bounded and compact.

Inspired by the previous papers, here we improve and complement some results contained there by using a completely different technique: we focus on an auxiliary approximating problem, show that the latter has solutions by variational arguments, and finally prove that these solutions converge to the ones of the original problem. Our approach requires to study separately the cases $b>S^{-2}$ and $b\in(0,S^{-2})$, and this is done in the following Sections \ref{sec:b>S-2} and \ref{sec:b<S-2}, respectively.

In the first case, we prove that the energy corresponding to \eqref{problem} has a global minimizer for $\lambda>0$ sufficiently large (Theorem \ref{existenceglobalmin}). As recalled before, such solution is added to the other one derived in \cite{f}, still for $\lambda$ large enough, by a mountain pass type approach. As a result, problem \eqref{problem} turns out to have at least two different solutions for $\lambda$ large enough and for $b\in(S^{-2},+\infty)$.

The case $0<b<S^{-2}$ is more delicate to handle. We introduce first some functionals approximating the critical Kirchhoff term, and work on suitable subsets of the approximating Nehari manifold. By making also use of the concentration-compactness principle, we prove the existence of one solution to \eqref{problem} for small values of $b$ (Theorem \ref{minim5}) and for small values of $\lambda$ (Theorem \ref{minim5bis}).
We stress out, that, unlike Theorem \ref{thmnaishi} in which the thresholds $b_*$ and $\lambda_*$ for which one has solutions are not specified, our approach permits to get explicit estimates of such constants. More specifically, in Theorem \ref{minim5} we show the existence of one solution to \eqref{problem} for $\displaystyle b\in\left(0,\frac{(p-2)^2}{4}S^{-2}\right)$ and for any $\lambda\in(0,+\infty)$; in Theorem \ref{minim5bis}, by using the same procedure, we obtain solutions for any $b\in(0,S^{-2})$ and for $\lambda\in \left(0,\Lambda_p\right)$, where
$$
\Lambda_p:=\frac{2a^{\frac{4-p}{2}}}{c_p^p(4-p)} \left(\frac{p-2}{p}\right)^{\frac{p-2}{2}}(S^{-2}-b)^{\frac{p-2}{2}}
$$
(see below for the definition of the constant $c_p$). We point out that, as in Theorem \ref{thmnaishi}, our constants $b_*$ and $\lambda_*$ are not "symmetric"; $b_*$ is independent of $\lambda$ (but depends on $p$), while $\lambda_*$ is a function of $b$. Moreover, our results, in both cases $b>S^{-2}$ and $b<S^{-2}$ are independent of the constant $a$ which can range all over the positive reals (in \cite{naishi2020existence} it is instead constantly assumed that $a=1$).

\subsection{Notation}
We will denote by $I_\lambda$ the energy functional corresponding to \eqref{problem}, i.e. the functional $I_\lambda:\W\to\R$ defined by
\begin{equation}
	I_\lambda(u) := \frac{a}{2}\|u\|^2+\frac{b}{4}\|u\|^4-\frac{1}{4}\|u\|_4^4-\frac{\lambda}{p}\|u\|_p^p,
\end{equation}
for each $u\in \W$, where
\begin{align*}
	\|u\|_m & := \ \left(\int_\Omega |u|^mdx\right)^{\frac{1}{m}}, \quad \text{for each } m\in [1,4],\\
	\|u\| &:= \|\nabla u\|_2.
\end{align*}
Moreover, we set
$$
c_m := \sup_{u\in \W\setminus\{0\}}\frac{\|u\|_m}{\|u\|}, \quad \text{for each } m\in [1,4];
$$
in the light of \eqref{defS-2}, one has
$$
S = c_4^{-2}.
$$
By a weak solution to \eqref{problem} we mean any $u\in\W$ such that
$$
\left( a+b \int_\Omega |\nabla u|^2 dx\right)\int_\Omega \nabla u \nabla v dx -\int_\Omega u^3 v dx -\lambda\int_\Omega u^{p-1}v dx=0,
$$
for all $v\in\W$. Since our approach to \eqref{problem} relies upon approximating procedures, from now on we fix a sequence $\{\sigma_n\}\subset (0,4-p)$, with $\displaystyle\lim_{n\rightarrow +\infty}\sigma_n=0$, and define the approximating energy functional by
\begin{equation}
	I_{\lambda,n}(u):=\frac{a}{2}\|u\|^2+\frac{b}{4-\sigma_n}\|u\|^{4-\sigma_n}-\frac{1}{4-\sigma_n}\|u\|_{4-\sigma_n}^{4-\sigma_n}-\frac{\lambda}{p}\|u\|_p^p,
\end{equation}
for each $u\in \W$ and $n\in \N$.

\section{The case $b>S^{-2}$}\label{sec:b>S-2}
Our existence result for problem \eqref{problem} relies upon the following auxiliary ones.

\begin{lemma}\label{coer}
	There exists $n_0\in \N$ such that
	\begin{eqnarray}\label{cIn}
		\lim_{\|u\|\rightarrow +\infty}\inf_{n\geq n_0} I_{\lambda,n}(u)=+\infty.
	\end{eqnarray}
\end{lemma}

\begin{proof}
\;	Let $n\in \N$ and $u\in \W$. Then,
	\begin{align*}
		I_{\lambda,n}(u)& \geq \frac{b}{4-\sigma_n}\|u\|^{4-\sigma_n} -
		\frac{|\Omega|^{\frac{\sigma_n}{4}}}{4-\sigma_n}\|u\|_4^{4-\sigma_n}-\frac{\lambda c_p^{p}}{p}\|u\|^p \\
		&\geq \left(\frac{b}{4-\sigma_n}-\frac{S^{-\frac{4-\sigma_n}{2}}|\Omega|^{\frac{\sigma_n}{4}}}{4-\sigma_n}\right)\|u\|^{4-\sigma_n}-\frac{\lambda
			c_p^{p}}{p}\|u\|^p.
	\end{align*}
	Since $b>S^{-2}$ and $\sigma_n\rightarrow 0$ as $n\to +\infty$, one has
	\begin{equation*}
		\lim_{n\rightarrow +\infty}
		\left(\frac{b}{4-\sigma_n}-\frac{S^{-\frac{4-\sigma_n}{2}}|\Omega|^{\frac{\sigma_n}{4}}}{4-\sigma_n}\right)=\frac{b-S^{-2}}{4}>0.
	\end{equation*}
	Moreover, recalling that $2<p<4$, we have
	\begin{eqnarray*}
		\lim_{n\rightarrow +\infty}(4-\sigma_n)=4>2+\frac{p}{2}>p,
	\end{eqnarray*}
	and therefore there exists $n_0\in \N$ such that
	\begin{align*}
		I_{\lambda,n}(u)&\geq \frac{b-S^{-2}}{8}\|u\|^{4-\sigma_n}-\frac{\lambda c_p^{p}}{p}\|u\|^p \geq \frac{b-S^{-2}}{8}(\|u\|^{2+\frac{p}{2}}-1)-\frac{\lambda c_p^{p}}{p}\|u\|^p,
	\end{align*}
	for each $n\in\N$ with $n\geq n_0$, from which \eqref{cIn} follows.
\end{proof}

\medskip

Without loss of generality, we may assume that $n_0=1$ in \eqref{cIn}.

\begin{lemma}\label{Ininf}
	There exists $\lambda_0>0$ such that
	\begin{equation}
		\sup_{n\in \N}\inf_{u\in \W}I_{\lambda,n}(u)<0, \quad\text{for  each } \lambda>\lambda_0.
	\end{equation}
\end{lemma}

\begin{proof}
\;	Fix $u_0\in \W\setminus \{0\}$, with $\|u_0\|=1$. Recalling that $0<\sigma_n<4-p$, one has
	\begin{equation*}
		I_{\lambda,n}(u_0) \leq \frac{a}{2}+\frac{b}{p}-\frac{\lambda}{p}\|u_0\|_p^p, \quad \text{for each } n\in\N.
	\end{equation*}
	Therefore,
	\begin{equation*}
		\sup_{n\in \N}\inf_{u\in \W}I_{\lambda,n}(u)\leq\frac{a}{2}+\frac{b}{p}-\frac{\lambda}{p}\|u_0\|_p^p<0,
	\end{equation*}
	for $\displaystyle\lambda>\lambda_0:=\frac{p}{\|u_0\|_p^p}\left(\frac{a}{2}+\frac{b}{p}\right)$.
\end{proof}

\medskip

\begin{theorem}\label{existenceglobalmin}
	Let $\lambda_0$ be as in Lemma \ref{Ininf}. Then, for each $\lambda>\lambda_0$ problem \eqref{problem} has at least one solution, which is a global minimizer of $I_\lambda$.
\end{theorem}

\begin{proof}
\; Fix $\lambda>\lambda_0$ and $n\in \N$. From Lemmas \ref{coer} and \ref{Ininf} we infer that $I_{\lambda,n}$ admits a global minimum point $u_n\in \W\setminus\{0\}$,  which we may assume non-negative in $\Omega$. Therefore, by the Strong Maximum Principle, $u_n>0$ in $\Omega$.
	
	By means again of Lemma \ref{coer} and \ref{Ininf} it is easy to check that the sequence $\{u_n\}$ is bounded in $\W$. Hence, there exist $l\in [0,+\infty)$ and a non-negative $u_\lambda \in \W$ such that, up to a subsequence,
	\begin{equation}\label{relationunul}
		\begin{split}
			& u_n\rightharpoonup u_\lambda \quad \text {in }\W;\\
			& \|u_n\|\rightarrow l;\\
			& u_n\rightarrow  u_\lambda \quad \text{in } L^m(\Omega) \text{ for each } m\in [1,4);\\
			& \text{ for each } m\in [1,4) \text{ there exists } g\in L^1(\Omega) \text{ such that } |u_n|^m\leq g \text{ a.e. in } \Omega;\\
			& u_n\rightarrow u_\lambda \text{ a.e. in }\Omega,
		\end{split}
	\end{equation}
	as $n\to +\infty$. As a result, by the dominated convergence theorem,
	\begin{align}
		\int_\Omega u_n(x)^{3-\sigma_n}\varphi(x)dx  & \rightarrow \int_\Omega u_\lambda(x)^3\varphi(x)dx,\label{lim1}\\
		\int_\Omega u_n(x)^{p-1}\varphi(x)dx & \rightarrow \int_\Omega u_\lambda(x)^{p-1}\varphi(x)dx,\label{lim2}
	\end{align}
	as $n\to+\infty$, for each $\varphi\in C_0^1(\overline{\Omega})$. Moreover, being $I_{\lambda,n}'(u_n)=0$, one has
	\begin{eqnarray*}
		&0&=I{'}_{\lambda,n}(u_n)(\varphi)\\
		& &=(a+b\|u_n\|^{2-\sigma_n})\int_\Omega \nabla u_n(x)\nabla \varphi(x)dx-\int_\Omega\left(u_n(x)^{3-\sigma_n}+\lambda
		u_n(x)^{p-1}\right)\varphi(x)dx;
	\end{eqnarray*}
	passing to the limit as $n\to +\infty$ and taking \eqref{lim1} and \eqref{lim2} into account, we get
	\begin{equation}\label{critical}
		0=(a+bl^2)\int_\Omega \nabla u_\lambda(x)\nabla \varphi(x)dx-\int_\Omega\left(u_\lambda(x)^{3}+\lambda u_\lambda(x)^{p-1}\right)\varphi(x)dx.
	\end{equation}
	Since $C_0^1(\overline{\Omega})$ is dense in $\W$, the previous inequality actually holds for each \linebreak  $\varphi \in \W$. In particular, by
	choosing $\varphi=u_\lambda$ one infers that
	\begin{equation*}
		0=(a+bl^2)\|u_\lambda\|^2-\|u_\lambda\|_4^4-\lambda \|u_\lambda\|_p^p,
	\end{equation*}
	from which
	\begin{eqnarray}\label{ul}
		\|u_\lambda\|_4^4=(a+bl^2)\|u_\lambda\|^2-\lambda\|u_\lambda\|_p^p.
	\end{eqnarray}
	Now, let $n\in \N$. Using \eqref{ul}, the identity
	\begin{equation*}
		\|u_n\|_{4-\sigma_n}^{4-\sigma_n}=\|u_n\|_{4-\sigma_n}^{4-\sigma_n}+I_{\lambda,n}'(u_n)(u_n)=a\|u_n\|^2+b\|u_n\|^{4-\sigma_n}-\lambda \|u_n\|_p^p,
	\end{equation*}
	and the fact that $u_n$ is a global minimizer for $I_{\lambda,n}$, we deduce
	\begin{align*}
		I_{\lambda,n}(u_n) & =a\left(\frac{1}{2}-\frac{1}{4-\sigma_n}\right)\|u_n\|^2-\lambda\left(\frac{1}{p}-\frac{1}{4-\sigma_n}\right)\|u_n\|_p^p\\
		&\leq
		\frac{a}{2}\|u_\lambda\|^2+\frac{b}{4-\sigma_n}\|u_\lambda\|^{4-\sigma_n}-\frac{1}{4-\sigma_n}\|u_\lambda\|_{4-\sigma_n}^{4-\sigma_n}-\frac{\lambda}{p}\|u_\lambda\|_p^p\\
		& = I_{\lambda,n}(u_\lambda).
	\end{align*}
	Taking the limit as $n\to+\infty$ and noticing that
	$$
	u_\lambda^{4-\sigma_n}\leq 1+u_\lambda^4\in L^1(\Omega), \quad u_n^{4-\sigma_n}\rightarrow u_\lambda^4 \quad\text{a.e. in } \Omega \text{ as } n\to+\infty,
	$$
	one gets
	\begin{equation*}
		\frac{a}{4}l^2-\lambda\left(\frac{1}{p}-\frac{1}{4}\right)\|u_\lambda\|_p^p\leq
		\frac{a}{2}\|u_\lambda\|^2+\frac{b}{4}\|u_\lambda\|^4-\frac{1}{4}\|u_\lambda\|_4^4-\frac{\lambda}{p}\|u_\lambda\|_p^p.
	\end{equation*}
	Moreover, by \eqref{ul} it follows that
	\begin{equation*}
		\frac{a}{4}l^2-\lambda\left(\frac{1}{p}-\frac{1}{4}\right)\|u_\lambda\|_p^p\leq
		\frac{a}{4}\|u_\lambda\|^2+\frac{b}{4}\left(\|u_\lambda\|^2-l^2\right)\|u_\lambda\|^2-\lambda\left(\frac{1}{p}-\frac{1}{4}\right)\|u_\lambda\|_p^p,
	\end{equation*}
	from which
	\begin{equation*}
		\frac{a}{4}l^2\leq \frac{a}{4}\|u_\lambda\|^2+\frac{b}{4}(\|u_\lambda\|^2-l^2)\|u_\lambda\|^2.
	\end{equation*}
	Then, it must be
	\begin{equation*}
		\|u_\lambda\|\geq l=\lim_{n\rightarrow + \infty}\|u_n\|.
	\end{equation*}
	By the first convergence in \eqref{relationunul}, the previous inequality implies that $u_n\rightarrow u_\lambda$ in $\W$ and, in particular, that
	\begin{equation*}
		l=\lim_{n\rightarrow \infty}\|u_n\|=\|u_\lambda\| \quad\text{and}\quad \lim_{n\rightarrow \infty}\|u_n\|_{4-\sigma_n}=\|u_\lambda\|_4.
	\end{equation*}
	Hence, by \eqref{critical} we obtain that $u_\lambda$ is a critical point of  $I_\lambda$. Since $\lambda>\lambda_0$, we also have
	\begin{equation*}
		I_\lambda(u_\lambda)=\lim_{n\rightarrow +\infty}I_{\lambda,n}(u_n)<0,
	\end{equation*}
	which means that $u_\lambda \neq 0$. In particular, $u_\lambda>0$ by the Strong Maximum Principle.  Finally, note that
	\begin{equation*}
		I_\lambda(u_\lambda)=\lim_{n\rightarrow +\infty}I_{\lambda,n}(u_n)\leq \lim_{n\rightarrow +\infty}I_{\lambda,n}(u)=I_\lambda(u), \quad \text{for all } u\in \W,
	\end{equation*}
	that is, $u_\lambda$ is a global minimizer for $I_\lambda$.
\end{proof}

\begin{remark}
	{\rm As anticipated in the Introduction, the existence of a solution for a problem more general than \eqref{problem} was recently obtained in \cite{f} still for large values of $\lambda$. Indeed, it is straightforward to verify that the functions
		$$
		M(t):= a+bt, \;\; t\geq 0, \quad\text{and}\quad f(t):= \max\{0,t\}^{p-1}, \;\; t\in\R,
		$$
		satisfy the set of assumptions $(M_1)-(M_2)$ and $(f_1)-(f_2)$ of \cite[Theorem 1.1]{f}. As a result, there exists $\lambda_1>0$ such that $I_\lambda$ has a mountain pass solution $v_\lambda\in \W$ for any $\lambda\geq\lambda_1$.
		Then, in the light of Theorem \ref{existenceglobalmin}, we can state that for $\lambda>\lambda^*:=\max\{\lambda_0,\lambda_1\}$, problem \eqref{problem} has at least two different solutions.
	}
\end{remark}

\section{The case $b<S^{-2}$}\label{sec:b<S-2}

To study problem \eqref{problem} in this setting, we introduce suitable approximating functionals and, associated with them, suitable subsets of the Nehari manifold in which we are going to work.

For each $n\in \N$, let $\Psi_{\lambda,n} J_{\lambda,n},K_{\lambda,n}:\W\rightarrow \R$ be the functionals defined by
\begin{equation}\label{funzpsijk}
	\begin{split}
		\Psi_{\lambda,n}(u)& :=\|u\|_{4-\sigma_n}^{4-\sigma_n}+\lambda \|u\|_p^p,\\
		J_{\lambda,n}(u) &:=\frac{a}{2}f(\Psi_{\lambda,n}(u))^2+\frac{b}{4-\sigma_n}f(\Psi_{\lambda,n}(u))^{4-\sigma_n}-\frac{1}{4-\sigma_n}\|u\|_{4-\sigma_n}^{4-\sigma_n}-\frac{\lambda}{p}\|u\|_p^p,\\
		K_{\lambda,n}(u) &:= 2af(\Psi_{\lambda,n}(u))^2+b(4-\sigma_n)f(\Psi_{\lambda,n}(u))^{4-\sigma_n}-(4-\sigma_n)\|u\|_{4-\sigma_n}^{4-\sigma_n}-\lambda
		p\|u\|_p^p,
	\end{split}
\end{equation}
for each $u\in \W$, where $f:(0,+\infty)\rightarrow (0,+\infty)$ is the inverse function of
\begin{equation*}
	(0,+\infty)\ni x\mapsto ax^2+bx^{4-\sigma_n}.
\end{equation*}
Notice that $f\in C^\infty(0,+\infty)$ and one has
\begin{equation*}
	f'(y)=\frac{1}{2af(y)+(4-\sigma_n)bf(y)^{3-\sigma_n}}, \quad \text{ for all } y\in (0,+\infty).
\end{equation*}
It is easy to see that the functions $y\mapsto f(y)^2$ and $y\mapsto f(y)^{4-\sigma_n}$ are $C^1$ in $[0,+\infty)$ and, consequently, being $\Psi_{\lambda,n}\in C^1(\W)$, the functionals $J_{\lambda,n}$ and $K_{\lambda,n}$ are $C^1$ in $\W$ as well, with derivatives given by
\begin{equation}\label{Jder}
	\begin{split}
		J_{\lambda,n}'(u)(\varphi)  &= \left(af(\Psi_{\lambda,n}(u))+bf(\Psi_{\lambda,n}(u))^{3-\sigma_n}\right)f'(\Psi_{\lambda,n}(u))\Psi'_{\lambda,n}(u)(\varphi)\\
		&\;\;\;-\int_\Omega
		\left(u(x)^{3-\sigma_n}+\lambda u(x)^{p-1}\right)\varphi(x) dx\\
		&=\frac{a+bf(\Psi_{\lambda,n}(u))^{2-\sigma_n}}{2a+(4-\sigma_n)bf(\Psi_{\lambda,n}(u))^{2-\sigma_n}}\int_\Omega
		\left((4-\sigma_n)u(x)^{3-\sigma_n}+\lambda p
		u(x)^{p-1}\right)\varphi(x)dx\\
		& \;\;\; -\int_\Omega \left(u(x)^{3-\sigma_n}+\lambda u(x)^{p-1}\right)\varphi(x)dx\\
		&=\int_\Omega\left\{\left[(4-\sigma_n)g_1(\Psi_{\lambda,n}(u))-1\right]u(x)^{3-\sigma_n}\right. \\
		&\;\;\;\left.  + \lambda\left[pg_1(\Psi_{\lambda,n}(u))-1\right]u(x)^{p-1}\right\}\varphi(x)dx,
	\end{split}
\end{equation}
and
\begin{equation}\label{Kder}
	\begin{split}
		K_{\lambda,n}'(u)(\varphi)
		&= \left(4af(\Psi_{\lambda,n}(u))+(4-\sigma_n)^2bf(\Psi_{\lambda,n}(u))^{3-\sigma_n}\right)f'(\Psi_{\lambda,n}(u))\Psi'_{\lambda,n}(u)(\varphi)\\
		&\;\;\;-\int_\Omega
		\left((4-\sigma_n)^2u(x)^{3-\sigma_n}+p^2u(x)^{p-1}\right)\varphi(x)dx\\
		&=\frac{4a+(4-\sigma_n)^2bf(\Psi_{\lambda,n}(u))^{2-\sigma_n}}{2a+(4-\sigma_n)bf(\Psi_{\lambda,n}(u))^{2-\sigma_n}}\int_\Omega
		\left((4-\sigma_n)u(x)^{3-\sigma_n}+\lambda p u(x)^{p-1}\right)\varphi(x) dx\\
		&\;\;\; -\int_\Omega
		\left((4-\sigma_n)^2u(x)^{3-\sigma_n}+\lambda p^2u(x)^{p-1}\right)\varphi dx\\
		&=\int_\Omega\left\{(4-\sigma_n)\left[g_2(\Psi_{\lambda,n}(u))-(4-\sigma_n)\right]u(x)^{3-\sigma_n}\right.\\[2mm]
		& \;\;\;\left.+\lambda p\left[g_2(\Psi_{\lambda,n}(u))- p\right]u(x)^{p-1}\right\}\varphi(x) dx,
	\end{split}
\end{equation}
for each $u,\varphi\in \W$, where
\begin{equation}\label{g1g2}
	g_1(y) :=\frac{a+bf(y)^{2-\sigma_n}}{2a+(4-\sigma_n)bf(y)^{2-\sigma_n}} \quad\text{and} \quad
	g_2(y) :=\frac{4a+(4-\sigma_n)^2bf(y)^{2-\sigma_n}}{2a+(4-\sigma_n)bf(y)^{2-\sigma_n}},
\end{equation}
for  $y\in [0,+\infty)$. By \eqref{Jder} and \eqref{Kder} above we can derive the following result, which we will use to apply Lagrange Multiplier Theorem.

\begin{lemma}\label{criticalKJ}
	The functionals $J_{\lambda,n}$ and $K_{\lambda,n}$ in \eqref{funzpsijk} have no non-zero critical point.
\end{lemma}

\begin{proof}
\;	Let $u\in \W$ be a critical point of $J_{\lambda,n}$. By \eqref{Jder} we infer that
	\begin{equation}\label{eqcrit}
		\left[(4-\sigma_n)g_1(\Psi_{\lambda,n}(u))-1)\right]u(x)^{3-\sigma_n}+\lambda\left[pg_1(\Psi_{\lambda,n}(u))-1\right]u(x)^{p-1}=0 \quad \text{ for a.e. } x\in\Omega.
	\end{equation}
	Since $p<4-\sigma_n$, the system
	\begin{equation*}
		\left\lbrace
		\begin{array}{l}
			(4-\sigma_n)g_1(y)-1=0\\
			pg_1(y)-1=0
		\end{array}
		\right.,
	\end{equation*}
	has no solution in $y\in[0,+\infty)$.  Therefore, by \eqref{eqcrit} it follows that either $u(x)\equiv 0$ or $u(x)\neq 0$ and, as a result,
	\begin{equation*}
		(4-\sigma_n)g_1(\Psi_{\lambda,n}(u))-1\neq 0,
	\end{equation*}
	from which, again by \eqref{eqcrit},
	\begin{equation*}
		u(x)^{4-p-\sigma_n}=-\lambda\frac{pg_1(\Psi_{\lambda,n}(u))-1}{(4-\sigma_n)g_1(\Psi_{\lambda,n}(u))-1}.
	\end{equation*}
	In particular, $\nabla u=0$ a.e. in $\Omega$, that is $u=0$. Arguing as above, since the system
	\begin{equation*}
		\left\lbrace
		\begin{array}{l}
			g_2(y)-4+\sigma_n=0,\\
			g_2(y)-p=0,
		\end{array}
		\right.
	\end{equation*}
	has no solution in $[0,+\infty)$ as well, then $K_{\lambda,n}$ has no non-zero critical point either and the proof is concluded.
\end{proof}

Now, for each $n\in \N$, let us fix $\xi_n\in (0,\sigma_n)$  and consider the sets
\begin{align*}
	\mathcal{N}_n &:=\left\{u\in \W\setminus\{0\}: I_{\lambda,n}'(u)(u)=a\|u\|^2+b\|u\|^{4-\sigma_n}-\|u\|_{4-\sigma_n}^{4-\sigma_n}-\lambda
	\|u\|_p^p=0\right\}, \\
	\mathcal{A}_n & := \left\{u\in \W\setminus\{0\}: I_{\lambda,n}'(u)(u)\leq 0\ \text{ and }\ K_{\lambda,n}(u)\leq 0\right\}, \\
	\mathcal{C}_{n,R} & :=\left\{u\in \W\setminus\{0\}: \|u\|_{4-\xi_n}\leq R\right\}, \quad \text{ for all } R>0.
\end{align*}

The next lemma provides an estimate of the infimum of $I_{\lambda,n}$ on the intersection of the above sets.
\begin{lemma}\label{mp1}
	There exist $n_0\in \N$ and $\gamma,R_0>0$  such that $\mathcal{N}_n\cap \mathcal{A}_n\cap \mathcal{C}_{n,R_0}\neq \emptyset$ and
	\begin{equation}
		\inf_{\mathcal{N}_n\cap \mathcal{A}_n\cap \mathcal{C}_{n,R_0}} I_{\lambda,n}<\frac{a^2}{4(S^{-2}-b)}-\gamma
	\end{equation}
	for each $n\geq n_0$.
\end{lemma}

\begin{proof}
\;	Let $\varphi\in C_0^\infty(\Omega)$ be a non-negative function such that $\varphi\equiv 1$ in a closed ball $\overline{B}_R(x_0)$ centered at $x_0\in\Omega$ and of radius $R>0$ contained in $\Omega$.  For
	each $\varepsilon>0$, consider the function $v_\varepsilon:\Omega\rightarrow \R$ defined by
	\begin{equation}
		v_\varepsilon(x):=\frac{\varphi(x)}{\varepsilon+|x-x_0|^2}, \quad\text{for each } x\in \Omega,
	\end{equation}
	and set
	\begin{equation}
		u_\varepsilon :=\frac{v_\varepsilon}{\|v_\varepsilon\|}.
	\end{equation}
	It is well-known (see, for instance, \cite{bn}) that
	\begin{equation}\label{behave}
		\|v_\varepsilon\|^2=\frac{C_1}{\varepsilon}+O(1), \  \|v_\varepsilon\|_4^2=\frac{C_2}{\varepsilon}+O(\varepsilon), \ \text{and } \
		S^{-2}=S^{-2}\|u_\varepsilon\|\geq \|u_\varepsilon\|_4^4=S^{-2}+O(\varepsilon).
	\end{equation}
	A direct computation shows, moreover, that
	\begin{equation}\label{behave1}
		\|u_\varepsilon\|_p^p=C_3\varepsilon^{\frac{4-p}{2}}+O(\varepsilon^{\frac{p}{2}}),
	\end{equation}
	and thus
	\begin{equation}\label{behave2}
		\frac{\|v_\varepsilon\|_p^p}{\|v_\varepsilon\|^4}=\frac{\|u_\varepsilon\|_p^p}{\|v_\varepsilon\|^{4-p}}=
		\frac{C_3\varepsilon^{\frac{4-p}{2}}+O(\varepsilon^{\frac{p}{2}})}{\left(\displaystyle\frac{C_1}{\varepsilon}+O(1)\right)^{\frac{4-p}{2}}}=C_4\varepsilon^{4-p}+O(\varepsilon^2).
	\end{equation}
	Here, $C_i$, $i=1,\ldots,4$, are suitable positive constants independent of $\varepsilon$. Because of $b<S^{-2}$, \eqref{behave} and \eqref{behave2},  
	we can fix $\varepsilon_0,\eta_0>0$ such that
	\begin{equation}\label{up1}
		b+\eta_0=b\|u_\varepsilon\|^4+\eta_0<\|u_\varepsilon\|_4^4,
	\end{equation}
	and
	\begin{equation}\label{up2}
		\frac{a}{\|v_\varepsilon\|^2}+b\|u_\varepsilon\|^4+\eta_0<\|u_\varepsilon\|_4^4+\lambda\frac{\|v_\varepsilon\|_p^p}{\|v_\varepsilon\|^4},
	\end{equation}
	for all  $\varepsilon\in (0,\varepsilon_0)$. By \eqref{up1} and \eqref{up2} we get
	\begin{equation}
		b\|v_\varepsilon\|^4< \|v_\varepsilon\|_4^4 \quad \text{and}\quad a\|v_\varepsilon\|^2+b\|v_\varepsilon\|^4<\|v_\varepsilon\|_4^4+\lambda \|v_\varepsilon\|_p^p,
	\end{equation}
	for all  $\varepsilon\in (0,\varepsilon_0)$.  Since $\|v_\varepsilon\|^{4-\sigma_n}\rightarrow \|v_\varepsilon\|^4$ and
	$\|v_\varepsilon\|_{4-\sigma_n}^{4-\sigma_n}\rightarrow \|v_\varepsilon\|_4^4$ as $n\to+\infty$, for each $\varepsilon\in (0,\varepsilon_0)$ there exists $n_\varepsilon\in\N$ such that
	\begin{equation*}
		b\|v_\varepsilon\|^{4-\sigma_n}< \|v_\varepsilon\|_{4-\sigma_n}^{4-\sigma_n} \quad\text{and} \quad
		a\|v_\varepsilon\|^2+b\|v_\varepsilon\|^{4-\sigma_n}<\|v_\varepsilon\|_{4-\sigma_n}^{4-\sigma_n}+\lambda \|v_\varepsilon\|_p^p,
	\end{equation*}
	for each $n\geq n_\varepsilon$. Consider such an $n$. Since $2<p<4-\sigma_n$, one has
	\begin{equation*}
		a\|tv_\varepsilon\|^2+b\|tv_\varepsilon\|^{4-\sigma_n}>\|tv_\varepsilon\|_{4-\sigma_n}^{4-\sigma_n}+\lambda \|tv_\varepsilon\|_p^p
	\end{equation*}
	for $t>0$ small enough. Consequently, $t^*v_\varepsilon \in \mathcal{N}_n$, for some $t^*\in (0,1)$. This means that
	\begin{equation*}
		f\left(\Psi_{\lambda,n}(t^*v_{\varepsilon})\right)=\|t^*v_{\varepsilon}\|,
	\end{equation*}
	which, together with
	\begin{equation*}
		b\|t^*v_\varepsilon\|^{4-\sigma_n}< \|t^*v_\varepsilon\|_{4-\sigma_n}^{4-\sigma_n},
	\end{equation*}
	yields
	\begin{align*}
		K_{\lambda,n}(t^*v_\varepsilon) &=2af(\Psi_{\lambda,n}(t^*v_\varepsilon))^2+b(4-\sigma_n)f(\Psi_{\lambda,n}(t^*v_\varepsilon))^{4-\sigma_n}
		-(4-\sigma_n)\|t^*v_\varepsilon\|_{4-\sigma_n}^{4-\sigma_n}-\lambda
		p\|t^*v_\varepsilon\|_p^p\\
		&=a(2-p)\|t^*v_\varepsilon\|^2+(4-p-\sigma_n)\left(b\|t^*v_\varepsilon\|^{4-\sigma_n}-\|t^*v_\varepsilon\|_{4-\sigma_n}^{4-\sigma_n}\right)\\
		&<0,
	\end{align*}
	that is, $t^*v_\varepsilon\in \mathcal{N}_n\cap \mathcal{A}_n$.
	
	Now, notice that, for each $t>0$ and $\varepsilon\in (0,\varepsilon_0)$, one has
	\begin{equation}\label{Il}
		I_{\lambda}(tu_\varepsilon)=\frac{a}{2}t^2 +\frac{b-\|u_\varepsilon\|_4^4}{4}t^4-\lambda \frac{\|u_\varepsilon\|_p^p}{p}t^p\leq
		\frac{a}{2}t^2-\frac{\eta_0}{4}t^4.
	\end{equation}
	Moreover, if $\varepsilon \in (0,\varepsilon_0)$, the function
	\begin{equation*}
		(0,+\infty)\ni \tau\mapsto \frac{a}{2}\tau+\frac{b-\|u_\varepsilon\|_4^4}{4}\tau^2
	\end{equation*}
	attains its global maximum at
	\begin{equation*}\label{tau}
		\tau_\varepsilon:=\frac{a}{\|u_\varepsilon\|_4^4-b}\in \left[\frac{a}{S^{-2}-b},\frac{a}{\eta_0}\right),
	\end{equation*}
	and one has
	\begin{equation*}
		\max_{\tau>0}\left(\frac{a}{2}\tau+\frac{b-\|u_\varepsilon\|_4^4}{4}\tau^2\right)=\frac{a^2}{4(\|u_\varepsilon\|_4^4-b)}.
	\end{equation*}
	By using \eqref{Il}, we can find 
	$t_0\in \left(0,\sqrt{\frac{a}{S^{-2}-b}}\right)$ such that
	\begin{equation}\label{I2}
		I_{\lambda}(t u_\varepsilon)<\frac{a^2}{8(S^{-2}-b)}, \quad \text{for each } t\in (0, t_0) \text{ and } \varepsilon\in (0,\varepsilon_0),
	\end{equation}
	and, in addition,
	\begin{equation}\label{I1}
		I_{\lambda}(t u_\varepsilon)<-1, \quad \text{for each } t\geq t_1:=\sqrt{\frac{a+\sqrt{a+4\eta_0}}{\eta_0}} \text{ and } \varepsilon\in (0,\varepsilon_0).
	\end{equation}
	On the other hand, for every $t\in [t_0,t_1]$ and $\varepsilon \in (0,\varepsilon_0)$, we deduce the estimate
	\begin{equation}\label{I3}
		\begin{split}
			I_\lambda(tu_\varepsilon) &\leq \frac{a^2}{4(\|u_\varepsilon\|_4^4-b)}-\lambda \frac{t_0^p}{p}\|u_\varepsilon\|_p^p\\
			&=\frac{a^2}{4(S^{-2}-b)} +\frac{a^2}{4(\|u_\varepsilon\|_4^4-b)}-\frac{a^2}{4(S^{-2}-b)}-\lambda \frac{t_0^p}{p}\|u_\varepsilon\|_p^p.
		\end{split}
	\end{equation}
	By taking \eqref{behave}, \eqref{behave1} and the relations $\displaystyle 0<\frac{4-p}{2}<1<\frac{p}{2}$ into account, choosing $\varepsilon_0$ smaller if necessary, one has
	\begin{equation}\label{ep}
		\begin{split}
			&\frac{a^2}{4(\|u_\varepsilon\|_4^4-b)}-\frac{a^2}{4(S^{-2}-b)}-\lambda \frac{t_0^p}{p}
			\|u_\varepsilon\|_p^p \\
			&=\frac{a^2}{4(S^{-2}+O(\varepsilon)-b)}-\frac{a^2}{4(S^{-2}-b)}-\lambda C_3 \frac{{t_0}^p}{p}\varepsilon^{\frac{4-p}{2}}+O(\varepsilon^{\frac{p}{2}})\\
			&=O(\varepsilon)-\lambda C_3 \frac{{t_0}^p}{p}\varepsilon^{\frac{4-p}{2}}+O(\varepsilon^{\frac{p}{2}})\\
			&<0,
		\end{split}
	\end{equation}
	for each $\varepsilon \in (0,\varepsilon_0)$. Therefore, if we fix $\varepsilon \in (0,\varepsilon_0)$, by \eqref{I2}, \eqref{I1}, \eqref{I3} and \eqref{ep}, we get
	\begin{equation}\label{supIl}
		\sup_{t>0}I_{\lambda}(t u_\varepsilon)<\frac{a^2}{4(S^{-2}-b)}-\gamma,
	\end{equation}
	with
	$$
	\gamma=\frac{1}{2}\min\left\{\frac{a^2}{8(S^{-2}-b)},-\left( O(\varepsilon)-\lambda
	C_3\frac{{t_0}^p}{p}\varepsilon^{\frac{4-p}{2}}+O(\varepsilon^{\frac{p}{2}})\right)\right\}>0.
	$$
	Since $\|u_\varepsilon\|^{4-\sigma_n}=\|u_\varepsilon\|^4=1$ and $\|u_\varepsilon\|_{4-\sigma_n}^{4-\sigma_n}=\|u_\varepsilon\|_4^4$ as $n\to+\infty$, choosing possibly a bigger $n_\varepsilon\in\N$ we finally get
	\begin{equation*}
		\inf_{\mathcal{A}_n\cap \mathcal{N}_n\cap \mathcal{C}_{n,R_0}}I_{\lambda,n}\leq I_{\lambda,n}(t^*v_\varepsilon)\leq \max_{t>0}I_{\lambda,n}(t
		u_\varepsilon)<\frac{a^2}{4(S^{-2}-b)}-\gamma,
	\end{equation*}
	for  all  $n\geq n_\varepsilon$. The conclusion then follows by taking $n_0=n_\varepsilon$ and $R_0:=\sup_{n\in \N}\|v_\varepsilon\|_{4-\xi_n}.$
\end{proof}

\medskip

In the following, without loss of generality, we assume that $n_0=1$ and we fix a positive number $R$ such that
\begin{eqnarray}\label{R1}
	R\geq R_0 \quad \text{and} \quad R^2>\frac{2ap}{(S^{-2}-b)(p-2)^2}\max\left\{1,\frac{\sup_{n\in\N}|\Omega|^{\frac{\xi_n}{2(4-\xi_n)}}}{S}\right\},
\end{eqnarray}
where $R_0$ is as in Lemma \ref{mp1}.

\begin{lemma}\label{minim}
	Let $r_0$ be the unique positive solution of the equation
	\begin{equation}
		ax^2+(b-S^{-2})x^4-\lambda c_p^px^p=0,
	\end{equation}
	and let $R>0$ be as in $(\ref{R1})$. Then, for each $r\in (0,r_0)$, there exists $n_r\in \N$ with the following property: for each $n\geq n_r$
	there exists a non-negative function $u_n\in \mathcal{A}_n\cap
	\mathcal{N}_n\cap \mathcal{C}_{n,R}$ such that $\|u_n\|\geq r$ and
	\begin{equation}\label{Jlambdanun}
		J_{\lambda,n}(u_n)=\inf_{\mathcal{A}_n\cap \mathcal{C}_{n,R}}J_{\lambda,n}=\inf_{\mathcal{A}_n\cap \mathcal{N}_n\cap \mathcal{C}_{n,R}}J_{\lambda,n}
		\geq \frac{a(p-2)(2-\sigma_n)}{2p(4-\sigma_n)}\|u_n\|^2.
	\end{equation}
\end{lemma}

\begin{proof}
\;	At first, we show that $\mathcal{A}_n$ is sequentially weakly closed in $\W$. Let $\{u_k\}$ be a sequence in $\mathcal{A}_n$ weakly converging to
	$u^*\in \W$. Since the functional
	$$
	\W\ni u \mapsto I'_{\lambda,n}(u)(u)
	$$
	is sequentially weakly lower semicontinuous and the functional $K_{\lambda,n}$ is sequentially weakly continuous, one has
	\begin{equation}\label{sc}
		I_{\lambda,n}'(u^*)(u^*)\leq 0 \quad \text{and} \quad K_{\lambda,n}(u^*)\leq 0.
	\end{equation}
	Let us to show that $u^*\neq 0$.  Letting $r>0$ and $u\in \W$ with $s:=\|u\|<r$, one has
	\begin{equation}\label{ineq1}
		\begin{split}
			I_{\lambda,n}'(u)(u)&=as^2+bs^{4-\sigma_n}-\|u\|_{4-\sigma_n}^{4-\sigma_n}-\lambda\|u\|_p^p\\
			&\geq as^2+bs^{4-\sigma_n}-|\Omega|^{\frac{\sigma_n}{4}}\|u\|_{4}^{4-\sigma_n}-\lambda c_p^ps^p\\
			&\geq as^2+bs^{4-\sigma_n}- |\Omega|^{\frac{\sigma_n}{4}}S^{-\frac{4-\sigma_n}{2}}s^{4-\sigma_n}-\lambda c_p^ps^p\\
			&>0,
		\end{split}
	\end{equation}
	if $r$ is small enough. Therefore
	\begin{equation}\label{bl}
		\|u\|\geq r \quad\text{for each } u\in \mathcal{A}_n.
	\end{equation}
	In particular, for each $k\in \N$, one has
	$$
	\|u_k\|_{4-\sigma_n}^{4-\sigma_n}+\lambda \|u_k\|_p^p\geq a\|u_k\|^2+b\|u_k\|^{4-\sigma_n}\geq ar^2+br^{4-\sigma_n}>0,
	$$
	and taking the limit as $k\rightarrow +\infty$,
	$$
	\|u^*\|_{4-\sigma_n}^{4-\sigma_n}+\lambda \|u^*\|_p^p\geq ar^2+br^{4-\sigma_n}>0,
	$$
	which implies $u^*\neq 0$, as desired. This means, taking also \eqref{sc} into account, that $u^*\in \mathcal{A}_n$ and hence $\mathcal{A}_n$ is a sequentially weakly closed subset of $\W$.
	
	Now, it is easy to see that the set  $\mathcal{A}_n\cap \mathcal{C}_{n,R}$ is sequentially weakly closed and bounded, that is, weakly compact. Together with the sequential weak continuity of $J_{\lambda,n}$, this implies the existence of $u_n\in \mathcal{A}_n\cap \mathcal{C}_{n,R}$ such that
	$$
	J_{\lambda,n}(u_n)=\inf_{u\in \mathcal{A}_n\cap \mathcal{C}_{n,R}}J_{\lambda,n}(u).
	$$
	Clearly, we may assume $u_n\geq 0$. We now show that $u_n\in \mathcal{N}_n$. If not, then
	$$
	a\|u_n\|^2+b\|u_n\|^{4-\sigma_n}<\|u_n\|_{4-\sigma_n}^{4-\sigma_n}+\lambda \|u_n\|_p^p
	$$
	and $u_n$ would be a local minimizer of $J_{\lambda,n}$ restricted to the set
	$$
	\{u\in \W\setminus\{0\}: K_{\lambda,n}(u)\leq 0, \ \|u\|_{4-\xi_n}\leq R\}.
	$$
	Since the set
	$$
	\{u\in \W\setminus\{0\}: K_{\lambda,n}(u)=0, \ \|u\|_{4-\xi_n}=R\}
	$$
	is a $C^1$--manifold, having $K_{\lambda,n}$ (and $u\mapsto \|u\|_{4-\xi_n}^{4-\xi_n}-R^{4-\xi_n}$) no non-zero critical point (see Lemma
	\ref{criticalKJ}), by Lagrange Multiplier Theorem there should exist $\mu,\nu\in \R$ such that
	\begin{equation}\label{eq1}
		J_{\lambda,n}'(u_n)(u)+\mu K'_{\lambda,n}(u_n)(u)+\nu \int_\Omega u_n^{3-\xi_n}udx=0, \quad\text{for each } u\in \W.
	\end{equation}
	By \eqref{Jder}, \eqref{Kder} and \eqref{eq1} we would get
	\begin{equation}\label{eqg}
		\begin{split}
			&\left[(4-\sigma_n)g_1(\Psi_{\lambda,n}(u_n))-1+\mu(4-\sigma_n)(g_2(\Psi_{\lambda,n}(u_n))-4+\sigma_n)\right]u_n(x)^{3-\sigma_n} + \nu u_n(x)^{3-\xi_n} \\
			& + \lambda\left[p g_1(\Psi_{\lambda,n}(u_n))-1+\mu p(g_2(\Psi_{\lambda,n}(u_n))-p)\right]u_n(x)^{p-1}=0, \quad \text{ a.e. in } \Omega.
		\end{split}
	\end{equation}
	Denote by $\nu_1$ and $\nu_2$ the coefficients of $u_n(x)^{3-\sigma_n}$ and $u_n(x)^{p-1}$ in \eqref{eqg}, respectively. If $u_n(x)\neq 0$, then, by
	\eqref{eqg} $u_n(x)$ would solve the equation
	\begin{equation*}
		\nu_1 y^{4-p-\sigma_n}+\nu y^{4-p-\xi_n}+\nu_2=0
	\end{equation*}
	which has at most two solutions in $[0,+\infty)$ if at least one of the coefficients $\nu_1,\nu,\nu_2$ is non-zero. In this case, $u_n$ would attain at most three distinct values in $\Omega$, and then $\nabla u_n=0$ a.e. in $\Omega$, i.e. $u_n=0$, a contradiction. Consequently, it must be
	$\nu=0$,
	\begin{equation*}
		\begin{split}
			\nu_1 & =(4-\sigma_n)g_1(\Psi_{\lambda,n}(u_n))-1+\mu(4-\sigma_n)(g_2(\Psi_{\lambda,n}(u_n))-4+\sigma_n)=0,\\
			\nu_2 & = \lambda( p g_1(\Psi_{\lambda,n}(u_n))-1+\mu p(g_2(\Psi_{\lambda,n}(u_n))-p))=0
		\end{split}
	\end{equation*}
	that is
	\begin{equation*}
		\begin{split}
			g_1(\Psi_{\lambda,n}(u_n))+\mu g_2(\Psi_{\lambda,n}(u_n)) &=\mu(4-\sigma_n)+\frac{1}{4-\sigma_n},\\
			g_1(\Psi_{\lambda,n}(u_n))+\mu g_2(\Psi_{\lambda,n}(u_n))&=\mu p+\frac{1}{p}.
		\end{split}
	\end{equation*}
	By the above identities, we get $\mu=\frac{1}{p(4-\sigma_n)}$ and
	\begin{equation}\label{eqg1}
		g_1(\Psi_{\lambda,n}(u_n))+\frac{1}{p(4-\sigma_n)} g_2(\Psi_{\lambda,n}(u_n))=\frac{1}{4-\sigma_n}+\frac{1}{p}.
	\end{equation}
	Now, by  \eqref{g1g2}, it turns out that
	\begin{equation*}
		g_2(\Psi_{\lambda,n}(u_n))=
		(4-\sigma_n)^2g_1(\Psi_{\lambda,n}(u_n))-\frac{a((4-\sigma_n)^2-4)}{2a+(4-\sigma_n)bf(\Psi_{\lambda,n}(u_n))^{2-\sigma_n}}.
	\end{equation*}
	Plugging the above relation in \eqref{eqg1}, we get
	\begin{equation*}
		\left(1+\frac{4-\sigma_n}{p}\right)g_1(\Psi_{\lambda,n}(u_n))=\frac{1}{p(4-\sigma_n)} \cdot
		\frac{a((4-\sigma_n)^2-4)}{2a+(4-\sigma_n)bf(\Psi_{\lambda,n}(u_n))^{2-\sigma_n}}+\frac{4-\sigma_n+p}{p(4-\sigma_n)},
	\end{equation*}
	from which
	\begin{equation*}
		g_1(\Psi_{\lambda,n}(u_n))-\frac{1}{4-\sigma_n}=\frac{1}{(4-\sigma_n+p)(4-\sigma_n)} \cdot
		\frac{a((4-\sigma_n)^2-4)}{2a+(4-\sigma_n)bf(\Psi_{\lambda,n}(u_n))^{2-\sigma_n}}.
	\end{equation*}
	Finally, using the definition of $g_1$ in the above identity, we get
	$$
		\frac{(2-\sigma_n)a}{2a+(4-\sigma_n)bf(\Psi_{\lambda,n}(u_n))^{2-\sigma_n}}=\frac{1}{4-\sigma_n+p} \cdot
		\frac{a((4-\sigma_n)^2-4)}{2a+(4-\sigma_n) b f(\Psi_{\lambda,n}(u_n))^{2-\sigma_n}}
	$$
	that implies
	$$
	(2-\sigma_n)(4-\sigma_n+p)=(4-\sigma_n)^2-4.
	$$
	By dividing both sides by $2-\sigma_n$, one has
	$$
	4-\sigma_n+p=6-\sigma_n
	$$
	and so $p=2$, a contradiction. Hence, $u_n\in \mathcal{N}_n$ and, in particular,
	$$
	J_{\lambda,n}(u_n)=\inf_{u\in \mathcal{A}_n\cap \mathcal{C}_{n,R_0}}J_{\lambda,n}(u)=\inf_{u\in
		\mathcal{A}_n\cap\mathcal{N}_n\cap\mathcal{C}_{n,R_0}}J_{\lambda,n}(u).
	$$
	Let us finally show that the inequality in \eqref{Jlambdanun} holds. Let $r\in (0,r_0)$. By the definition of $r_0$, we can find $n_{r}>0$ such that
	\begin{equation*}
		as^2+bs^{4-\sigma_n}-|\Omega|^{\frac{\sigma_n}{4}}S^{-\frac{4-\sigma_n}{2}}s^{4-\sigma_n}-\lambda c_p^ps^p>0,
	\end{equation*}
	for each $n\geq n_{r}$ and $s\in (0,r)$. For such $n$ and $r$, in the light of \eqref{ineq1} and the fact that $u_n\in\mathcal{A}_n$, we deduce
	at once that $\left\| u_n\right\| \geq r$. If $u\in \mathcal{N}_n\cap \mathcal{A}_n$, then
	\begin{equation}\label{ineq2}
		a\|u\|^2+b\|u\|^{4-\sigma_n}=\Psi_{\lambda,n}(u)=\|u\|_{4-\sigma_n}^{4-\sigma_n}+\lambda \|u\|_p^p
	\end{equation}
	and, by the definition of $f$, one has
	\begin{equation}\label{fu1}
		f(\Psi_{\lambda,n}(u))=f(a\|u\|^2+b\|u\|^{4-\sigma_n})=\|u\|.
	\end{equation}
	In addition, by $K_{\lambda,n}(u)\leq 0$ one also has
	\begin{equation*}
		2a\|u\|^2+b(4-\sigma_n)\|u\|^{4-\sigma_n}\leq (4-\sigma_n)\|u\|_{4-\sigma_n}^{4-\sigma_n}+\lambda p \|u\|_p^p
	\end{equation*}
	and, by \eqref{ineq2},
	\begin{equation}\label{ineq3}
		(4-p-\sigma_n)(b\|u\|^{4-\sigma_n}-\|u\|_{4-\sigma_n}^{4-\sigma_n})\leq a(p-2)\|u\|^2.
	\end{equation}
	In conclusion, with the aid of \eqref{bl}, \eqref{ineq2}, \eqref{fu1} and \eqref{ineq3}, we get
	\begin{equation}\label{Ju0}
		\begin{split}
			J_{\lambda,n}(u)&=a\left(\frac{1}{2}-\frac{1}{p}\right)\|u\|^2 +\left(\frac{1}{4-\sigma_n}-\frac{1}{p}\right)\left(b
			\|u\|^{4-\sigma_n}- \|u\|_{4-\sigma_n}^{4-\sigma_n}\right)\\
			&= a\frac{p-2}{2p}\|u\|^2-\frac{4-p-\sigma_n}{p(4-\sigma_n)}\left(b
			\|u\|^{4-\sigma_n}- \|u\|_{4-\sigma_n}^{4-\sigma_n}\right)\\
			&\geq a\frac{p-2}{2p}\|u\|^2-a\frac{p-2}{p(4-\sigma_n)}\|u\|^2\\ 
			&\geq \frac{a(p-2)(2-\sigma_n)}{2p(4-\sigma_n)}\|u\|^2.
		\end{split}
	\end{equation}
	This concludes the proof.
\end{proof}

In what follows $r\in (0,r_0)$ is fixed and we will assume, without loss of generality, that $n_r=1$. Collecting the results from the previous two lemmas we obtain the following one.

\begin{lemma}\label{minim2}
	Let $\gamma>0$ be as in Lemma \ref{mp1} and $R>0$ be as in  (\ref{R1}). Then, if $n\in \N$, and
	$u_n\in \mathcal{A}_n\cap \mathcal{N}_n\cap \mathcal{C}_{n,R}$ is as in Lemma \ref{minim}, one has  $\|u_n\|\leq R$ and
	$$
	\frac{a(p-2)(2-\sigma_n)}{2p(4-\sigma_n)} r^2 \leq I_{\lambda,n}(u_n)=\inf_{\mathcal{A}_n\cap \mathcal{N}_n}I_{\lambda,n}<\frac{a^2}{4(S^{-2}-b)}-\gamma
	$$
\end{lemma}

\begin{proof}
\;	For each $u\in \mathcal{N}_n$ one has $f(\Psi_{\lambda,n}(u))=\|u\|$ and so $J_{\lambda,n}(u)=I_{\lambda,n}(u)$. Consequently,
	\begin{equation*}
		\inf_{u\in \mathcal{A}_n\cap \mathcal{N}_n}J_{\lambda,n}=\inf_{u\in \mathcal{A}_n\cap \mathcal{N}_n}I_{\lambda,n},
	\end{equation*}
	and, by Lemma \ref{minim},
	\begin{eqnarray}\label{JI}
		&&J_{\lambda,n}(u_n)=\inf_{u\in \mathcal{A}_n\cap \mathcal{N}_n\cap \mathcal{C}_{n,R}}J_{\lambda,n}=\inf_{u\in \mathcal{A}_n\cap \mathcal{N}_n\cap
			\mathcal{C}_{n,R}}I_{\lambda,n}=I_{\lambda,n}(u_n).
	\end{eqnarray}
	Moreover, by Lemmas \ref{mp1} and \ref{minim},
	\begin{eqnarray}\label{JI1}
		\frac{a(p-2)(2-\sigma_n)}{2p(4-\sigma_n)}r^2 \leq \inf_{ \mathcal{A}_n\cap \mathcal{N}_n\cap \mathcal{C}_{n,R}}J_{\lambda,n}=\inf_{\mathcal{A}_n\cap
			\mathcal{N}_n\cap \mathcal{C}_{n,R}}I_{\lambda,n}<\frac{a^2}{4(S^{-2}-b)}.
	\end{eqnarray}
	Now, recalling that $\sigma_n\in (0,4-p)$, by \eqref{Ju0} one has
	\begin{equation*}
		J_{\lambda,n}(u)\geq a\frac{(p-2)^2}{8p}\|u\|^2\geq a S\frac{(p-2)^2}{8p}\|u\|_4^2\geq a |\Omega|^{-\frac{\xi_n}{2(4-\xi_n)}}
		S\frac{(p-2)^2}{8p}\|u\|_{4-\xi_n}^2
	\end{equation*}
	for each $u\in \mathcal{A}_n\cap \mathcal{N}_n$. Therefore, by \eqref{JI1},
	\begin{eqnarray}\label{ineq7}
		a |\Omega|^{-\frac{\xi_n}{2(4-\xi_n)}}S\frac{(p-2)^2}{8p}\|u\|_{4-\xi_n}^2\leq a\frac{(p-2)^2}{8p}\|u\|^2<\frac{a^2}{4(S^{-2}-b)}
	\end{eqnarray}
	for each $u\in \mathcal{A}_n\cap \mathcal{N}_n$ such that
	\begin{equation}\label{ineq20}
		I_{\lambda,n}(u)=J_{\lambda,n}(u)<\frac{a^2}{4(S^{-2}-b)}
	\end{equation}
	Recalling the choice of $R$  (see (\ref{R1})), by $(\ref{ineq7})$ one has
	\begin{eqnarray*}
		\|u\|\leq R \ \ \ \ {\rm and } \ \ \ \ \|u\|_{4-\xi_n}\leq R
	\end{eqnarray*}
	for each $u\in \mathcal{A}_n\cap \mathcal{N}_n$ satisfying \eqref{ineq20} and so, in particular, for $u=u_n$. Consequently,
	\begin{equation*}
		\inf_{\mathcal{A}_n\cap \mathcal{N}_n}I_{\lambda,n}=\inf_{\mathcal{A}_n\cap \mathcal{N}_n\cap \mathcal{C}_{n,R}}I_{\lambda,n}.
	\end{equation*}
	The conclusion, then, follows from the previous equality, \eqref{JI} and \eqref{JI1}.
\end{proof}

\begin{lemma}\label{minim3}
	Assume that
	\begin{equation}\label{bsmall}
		b\leq\frac{(p-2)^2}{4}S^{-2}.
	\end{equation}
	Then, if $n\in \N$ is large enough and $u_n\in \mathcal{A}_n\cap \mathcal{N}_n$ is as in Lemma \ref{minim}, one has
	\begin{equation}\label{strictineq1}
		2a\|u_n\|^2+b(4-\sigma_n)\|u_n\|^{4-\sigma_n}<(4-\sigma_n)\|u_n\|^{4-\sigma_n}_{4-\sigma_n}+\lambda p \|u_n\|_p^p.
	\end{equation}
\end{lemma}

\begin{proof}
\;	Let $n\in \N$ and $u_n\in \mathcal{A}_n\cap\mathcal{N}_n$ be as in Lemma \ref{minim2}. One has $f(\Psi_{\lambda,n}(u))=\|u_n\|$. Then, since
	$u_n\in \mathcal{A}_n$, one has
	\begin{equation*}
		2a \|u_n\|^2+b(4-\sigma_n)\|u_n\|^{4-\sigma_n}\leq (4-\sigma_n)\|u_n\|^{4-\sigma_n}_{4-\sigma_n}+\lambda p \|u_n\|_p^p.
	\end{equation*}
	Arguing by contradiction, assume that the above relation holds as an equality for infinitely many $n$'s. So, $u_n$ satisfies
	\begin{equation*}
		\left\{\begin{array}{l}
			a\|u_n\|^2+b\|u_n\|^{4-\sigma_n}=\|u_n\|_{4-\sigma_n}^{4-\sigma_n}+\lambda  \|u_n\|_p^p \smallskip\\
			2a\|u_n\|^2 +b(4-\sigma_n)\|u_n\|^{4-\sigma_n}=(4-\sigma_n)\|u_n\|_{4-\sigma_n}^{4-\sigma_n}+\lambda p\|u_n\|_p^p \smallskip\\
			I_{\lambda,n}(u_n)<\displaystyle\frac{a^2}{4(S^{-2}-b)}-\gamma.
		\end{array}\right.
	\end{equation*}
	Then, we infer
	\begin{equation*}
		\left\{\begin{array}{l}
			a(p-2)\|u_n\|^2=(4-p-\sigma_n)(b\|u_n\|^{4-\sigma_n}-\|u_n\|_{4-\sigma_n}^{4-\sigma_n})\leq b(4-p-\sigma_n)\|u_n\|^{4-\sigma_n}\smallskip\\
			\displaystyle{\|u_n\|^2<\frac{4-\sigma_n}{2-\sigma_n}\cdot\frac{ap}{2(p-2)(S^{-2}-b)}-\gamma},
		\end{array}\right.
	\end{equation*}
	and therefore,
	\begin{equation*}
		\left(\frac{a(p-2)}{b(4-p-\sigma_n)}\right)^{\frac{2}{2-\sigma_n}}\leq \|u_n\|^2<\frac{4-\sigma_n}{2-\sigma_n}\cdot\frac{ap}{2(p-2)(S^{-2}-b)}-\gamma.
	\end{equation*}
	Passing to the limit as $n\rightarrow \infty$, we get
	\begin{equation*}
		\frac{a(p-2)}{b(4-p)}< \frac{ap}{(p-2)(S^{-2}-b)},
	\end{equation*}
	that is,
	$$
	b>\frac{(p-2)^2}{4}S^{-2},
	$$
	against \eqref{bsmall}.
\end{proof}

\medskip

The next lemma states that, if $b$ is below the threshold assumed in Lemma \ref{minim3}, then the approximating problem has a solution.

\begin{lemma}\label{minim4}
	Let $b\leq \displaystyle\frac{(p-2)^2}{4}S^{-2}$, $n\in \N$, and let $u_n\in \mathcal{A}_n\cap \mathcal{N}_n$ be as in Lemma \ref{minim2}.  Then,
	$I'_{\lambda,n}(u_n)=0$
\end{lemma}

\begin{proof}
\;	By Lemmas \ref{minim2} and \ref{minim3}, we infer that $u_n$ is a local minimum of $I_{\lambda,n}$ restricted to $\mathcal{N}_n$. Consider the $C^1$ functional $R_{\lambda,n}:\W\rightarrow \R$ defined by
	\begin{equation}
		R_{\lambda,n}(u):=I_{\lambda,n}'(u)(u), \quad\text{ for each } u\in \W.
	\end{equation}
	By Lemma \ref{minim3}, one has
	\begin{equation}\label{R}
		R_{\lambda,n}'(u_n)(u_n)= 2a\|u_n\|^2+b(4-\sigma_n)\|u_n\|^{4-\sigma_n}-(4-\sigma_n)\|u_n\|^{4-\sigma_n}_{4-\sigma_n}-\lambda p \|u_n\|_p^p<0.
	\end{equation}
	Thus, by the Lagrange Multipliers Theorem, there exists $\mu\in \R$ such that
	\begin{equation}\label{critical1}
		I'_{\lambda,n}(u_n)(\varphi)+\mu R_{\lambda,n}'(u_n)(\varphi)=0, \quad \text{ for each } \varphi\in \W.
	\end{equation}
	Testing the previous equality with $\varphi=u_n$, we get
	\begin{equation*}
		0=I'_{\lambda,n}(u_n)(u_n)+\mu R_{\lambda,n}'(u_n)(u_n)=\mu R_{\lambda,n}'(u_n)(u_n),
	\end{equation*}
	which, in view of \eqref{R},  implies $\mu=0$. By this and \eqref{critical1}, the conclusion follows.
\end{proof}

\medskip

We are finally in a position to show that \eqref{problem} has a solution.
\begin{theorem}\label{minim5}
	Let $b\leq \frac{(p-2)^2}{4}S^{-2}$ and, for each $n\in \N$, let $u_n\in \mathcal{A}_n\cap \mathcal{N}_n$ be as in Lemma \ref{minim2}.  Then, $\{u_n\}$
	strongly converges  in $\W$ to a solution $u^*$ of problem $(P_\lambda)$.
\end{theorem}

\begin{proof}
\;	By Lemma \ref{minim2}, we know that $\{u_n\}$ is a bounded sequence in $\W$ of non-negative functions. Therefore, up to a subsequence, we may assume
	that there exist a non-negative  $u_*\in \W$, and two measures $d\mu,d\nu$ such that
	\begin{equation}\label{conv}
		\left\{\begin{array}{ll}
			u_n\rightarrow u_* & \text{weakly in } \W;\\[2mm]
			u_n\rightarrow u_*  & \text{strongly in } L^m(\Omega), \text{ for each } m\in [1,4);\\[2mm]
			|\nabla u_n|^2 \rightarrow d\mu &\text{weakly}^*-\text{ in the sense of measures};\\[2mm]
			u_n^4 \rightarrow d\nu,  & \text{weakly}^*- \text{ in the sense of measures},
		\end{array}\right.
	\end{equation}
	with
	\begin{equation}\label{wsconv}
		\left\{\begin{array}{ll}
			d\mu\geq |\nabla u_*|^2+\displaystyle\sum_{k\in A}\mu_k\delta_{x_k};\\ \smallskip
			d\nu=  u_*^4+\displaystyle\sum_{k\in A}\nu_k\delta_{x_k};\\ \smallskip
			\mu_k^2 S^{-2}\geq \nu_k>0, \text{ for each } k\in A,
		\end{array}\right.
	\end{equation}
	where $A\subseteq \mathbb{N}$ is an at most countable set, and $x_k\in \overline{\Omega}$. Moreover, we may also assume that
	\begin{equation}\label{convergencetol}
		\lim_{n\rightarrow +\infty}\|u_n\|=:l\in [0,+\infty).
	\end{equation}
	Note that, by Lemma \ref{minim}, one has
	\begin{equation}\label{posl}
		l\geq r>0.
	\end{equation}
	Moreover, since $u_n\in \mathcal{A}_n\cap \mathcal{N}_n$, one has
	\begin{equation}\label{ineq5}
		a\|u_n\|^2+b\|u_n\|^{4-\sigma_n} = \|u_n\|_{4-\sigma_n}^{4-\sigma_n}+\lambda \|u_n\|_p^p
	\end{equation}
	and
	\begin{equation}\label{ineq5bis}
		2a\|u_n\|^2+b(4-\sigma_n)\|u_n\|^{4-\sigma_n}\leq (4-\sigma_n)\|u_n\|_{4-\sigma_n}^{4-\sigma_n}+\lambda p\|u_n\|_p^p.
	\end{equation}
	By \eqref{ineq5} and \eqref{ineq5bis}, we derive that
	\begin{equation}\label{ineq5ter}
		a\|u_n\|^2\geq \lambda \frac{4-p-\sigma_n}{2-\sigma_n}\|u_n\|_p^p
	\end{equation}
	and
	\begin{equation*}
		I_{\lambda,n}(u_n)=a\left(\frac{1}{2}-\frac{1}{4-\sigma_n}\right)\|u_n\|^2-\lambda\left(\frac{1}{p}-\frac{1}{4-\sigma_n}\right)\|u_n\|_p^p.
	\end{equation*}
	Then, taking \eqref{ineq5ter} and Lemma \ref{minim2} into account, it follows
	\begin{align}
		&a\left(\frac{1}{2}-\frac{1}{4-\sigma_n}\right)\|u_n\|^2-\lambda\left(\frac{1}{p}-\frac{1}{4-\sigma_n}\right)\|u_n\|_p^p  <\frac{a^2}{4(S^{-2}-b)}-\gamma,\label{ineq6}\\
		&a\frac{p-2}{2p}\cdot \frac{2-\sigma_n}{4-\sigma_n}\|u_n\|^2 <\frac{a^2}{4(S^{-2}-b)}-\gamma.\label{ineq6ter}
	\end{align}
	After that, passing to the limit as $n\rightarrow \infty$ in \eqref{ineq5} and \eqref{ineq6}, we get
	\begin{equation*}
		(a+bl^2)l^2\leq S^{-2}l^4+\lambda\|u_*\|_p^p \quad \text{and} \quad \frac{a}{4}l^2-\lambda \left(\frac{1}{p}-\frac{1}{4}\right)\|u_*\|_p^p\leq
		\frac{a^2}{4(S^{-2}-b)}-\gamma,
	\end{equation*}
	that is, in view of \eqref{posl},
	\begin{equation}\label{ineq10}
		l^2\geq \frac{a}{S^{-2}-b}- \frac{\lambda}{l^2(S^{-2}-b)}\|u_*\|_p^p  \quad \text{and} \quad l^2\leq \frac{a}{S^{-2}-b}-\frac{4\gamma}{a}+\lambda
		\frac{4-p}{ap}\|u_*\|_p^p,
	\end{equation}
	and this clearly implies that $u_*\neq 0$. Moreover, letting $n\to\infty$ in \eqref{ineq6ter}, we get
	\begin{equation}\label{ineql1}
		\frac{p-2}{p}l^2<\frac{a}{S^{-2}-b}.
	\end{equation}
	Now, if $\varphi \in C_0^\infty(\Omega)$, by Lemma \ref{minim4} and \eqref{conv}, one gets
	\begin{equation}\label{euleq}
		0=\lim_{n\rightarrow +\infty}I_{\lambda,n}'(u_n)(\varphi)=(a+bl^2)\int_\Omega \nabla u_*\nabla \varphi dx-\int_\Omega u_*^3\varphi dx-\lambda \int_\Omega
		u_*^{p-1}\varphi dx.
	\end{equation}
	By density, the above equality actually holds with for any $\varphi\in \W$. In particular, taking $\varphi=u_*$, we get
	\begin{equation}\label{ustar}
		0=(a+bl^2)\|u_*\|^2-\|u_*\|_4^4-\lambda \|u_*\|_p^p.
	\end{equation}
	Note also that, by \eqref{conv} and \eqref{wsconv}, one has
	\begin{equation}\label{ineql}
		l^2\geq \|u_*\|^2+ \sum_{k\in A}\mu_k.
	\end{equation}
	Now, assume that $A\neq \emptyset$ and fix $k\in A$. Fix also $\rho>0$ and let $\varphi_\rho\in C_0^\infty(\R^n)$ be such that
	\begin{equation*}
		\left\{
		\begin{array}{ll}
			0\leq \varphi_\rho(x)\leq 1 \;\text{ and }\; |\nabla \varphi_\rho(x)|\leq \displaystyle\frac{2}{\rho}  & \text{ for all } x\in \R^n;\\ \smallskip
			\varphi_\rho(x)= 0 & \text{ if } |x-x_k|\geq 2\rho;\\ \smallskip
			\varphi_\rho(x)= 1 & \text{ if } |x-x_k|\leq \rho.
		\end{array}
		\right.
	\end{equation*}
	Then, $u_n\varphi_\rho\in \W$ for each $n\in \N$, and since $I'_{\lambda,u_n}(u_n)=0$, we get
	\begin{equation}\label{eq2}
		\begin{split}
			0&=I'_{\lambda,n}(u_n)(u_n\varphi_\rho)\\
			&=(a+b\|u_n\|^{2-\sigma_n})\int_\Omega \left(|\nabla u_n|^2\varphi_\rho+u_n \nabla u_n \nabla \varphi_\rho
			\right)dx-\int_\Omega (u_n^{4-\sigma_n}+\lambda u_n^p)\varphi_\rho dx.
		\end{split}
	\end{equation}
	Moreover, using H\"older's inequality and the fact that $\sup_{\R^n}|\varphi_\rho|\leq 1$, one has
	\begin{equation*}
		\int_\Omega u_n^{4-\sigma_n}\varphi_\rho dx\leq |\Omega|^{\frac{\sigma_n}{4}}\left(\int_\Omega
		u_n^4\varphi_\rho^{\frac{4}{4-\sigma_n}}dx\right)^{\frac{4-\sigma_n}{4}}\leq |\Omega|^{\frac{\sigma_n}{4}}\left(\int_\Omega u_n^4\varphi_\rho
		dx\right)^{\frac{4-\sigma_n}{4}}
	\end{equation*}
	and
	\begin{equation*}
		\int_\Omega u_n \nabla u_n \nabla \varphi_\rho dx\leq \|u_n\|\left(\int_\Omega u_n^2|\nabla \varphi_\rho|^2dx\right)^{\frac{1}{2}},
	\end{equation*}
	which, together with \eqref{eq2}, imply
	\begin{equation}\label{ineq8}
		\begin{split}
			0& \geq (a+b\|u_n\|^{2-\sigma_n})\left[\int_\Omega |\nabla u_n|^2\varphi_\rho dx -\|u_n\|\left(\int_\Omega |u_n|^2|\nabla
			\varphi_\rho|^2dxdx\right)^{\frac{1}{2}}\right]\\
			&\;\;\;
			-|\Omega|^{\frac{\sigma_n}{4}}\left(\int_\Omega u_n^4\varphi_\rho dx\right)^{\frac{4-\sigma_n}{4}}-\lambda \int_{\Omega} u_n^p\varphi_\rho dx.
		\end{split}
	\end{equation}
	
	By \eqref{conv}, one has
	\begin{align*}
		\lim_{n\rightarrow +\infty}\int_\Omega |\nabla u_n|^2\varphi_\rho dx & = \int_\Omega \varphi_\rho d\mu \geq  \int_\Omega |\nabla u_*|^2 \varphi_\rho
		dx+\sum_{j\in A}\mu_j\varphi_\rho(x_j),\\ \smallskip
		\lim_{n\rightarrow +\infty}\|u_n\|\left(\int_\Omega u_n^2|\nabla \varphi_\rho|^2dx\right)^{\frac{1}{2}} & = l \left(\int_\Omega u_*^2|\nabla
		\varphi_\rho|^2dx\right)^{\frac{1}{2}}\\ \smallskip
		\lim_{n\rightarrow +\infty} \int_\Omega u_n^4\varphi_\rho dx  &= \int_\Omega u_*^4\varphi_\rho dx+\sum_{j\in A}\nu_j\varphi_\rho(x_j)\\ \smallskip
		\lim_{n\rightarrow +\infty}\int_\Omega u_n^p\varphi_\rho dx &=\int_\Omega u_*^p\varphi_\rho dx.
	\end{align*}
	Thus, passing to the limit as $n\rightarrow +\infty$ in \eqref{ineq8}, we get
	\begin{equation}\label{ineq9}
		\begin{split}
			0 & \geq (a+bl^2)\left[\int_\Omega |\nabla u_*|^2 \varphi_\rho dx+\sum_{j\in A}\mu_j\varphi_\rho(x_j) - l \left(\int_\Omega u_*^2|\nabla
			\varphi_\rho|^2dx\right)^{\frac{1}{2}}\right]\\
			&\;\;\;-\int_\Omega u_*^4\varphi_\rho dx-\sum_{j\in A}\nu_j\varphi_\rho(x_j)-\lambda \int_\Omega u_*^p\varphi_\rho dx.
		\end{split}
	\end{equation}
	Moreover, since
	\begin{align*}
		0 <l \left(\int_\Omega u_*^2|\nabla \varphi_\rho|^2dx\right)^{\frac{1}{2}} & = l \left(\int_{\Omega\cap \{|x-x_k|<2\rho\}} u_*^2|\nabla \varphi_\rho|^2dx\right)^{\frac{1}{2}}\\
		& \leq l \left(\int_{\Omega\cap \{|x-x_k|<2\rho\}} u_*^4 dx\right)^{\frac{1}{4}}\left(\int_{\Omega\cap \{|x-x_k|<2\rho\}}|\nabla
		\varphi_\rho|^4 dx\right)^{\frac{1}{4}}\\
		& \leq \frac{2l}{\rho}\cdot \left(\int_{\Omega\cap \{|x-x_k|<2\rho\}} u_*^4dx\right)^{\frac{1}{4}}
		\left(\int_{\{|x-x_k|<2\rho\}}dx\right)^{\frac{1}{4}}\\
		& = 2^{7/4}\pi^{1/2}l 
		\left(\int_{\Omega\cap \{|x-x_k|<2\rho\}} u_*^4dx\right)^{\frac{1}{4}} \to 0, \quad \text{as } \rho \rightarrow0,
	\end{align*}
	and
	\begin{align*}
		&\lim_{\rho\rightarrow 0}\int_\Omega |\nabla u_*|^2 \varphi_\rho dx=\lim_{\rho\rightarrow 0}\int_\Omega u_*^4\varphi_\rho dx=\lim_{\rho\rightarrow
			0}\int_\Omega u_*^p\varphi_\rho dx=0,\\
		&\lim_{\rho\rightarrow 0}\sum_{j\in A}\mu_j\varphi_\rho(x_j)=\mu_k,\\
		&\lim_{\rho\rightarrow 0}\sum_{j\in A}\nu_j\varphi_\rho(x_j)=\nu_k,
	\end{align*}
	taking the limit as $\rho\rightarrow 0$ in \eqref{ineq9}, one has
	$$
	0\geq (a+bl^2)\mu_k-\nu_k.
	$$
	Then, by the third inequality in \eqref{wsconv}, we infer
	$$
	0\geq (a+bl^2)\mu_k-S^{-2}\mu_k^2,
	$$
	and thus
	$$
	\mu_k\geq S^2(a+bl^2).
	$$
	By the above inequality and \eqref{ineql}, we get
	$$
	l^2\geq \|u_*\|^2+ S^2(a+bl^2)
	$$
	from which
	\begin{equation}\label{inequstar}
		l^2\geq \frac{a}{S^{-2}-b}+\frac{S^{-2}}{S^{-2}-b}\|u_*\|^2.
	\end{equation}
	Therefore, using the second inequality in \eqref{ineq10}, it follows that
	\begin{equation}\label{inequstar1}
		\frac{S^{-2}}{S^{-2}-b}\|u_*\|^2\leq \lambda \frac{4-p}{ap}\|u_*\|_p^p.
	\end{equation}
	Moreover, since by \eqref{ustar} one has
	\begin{equation*}
		\lambda \|u_*\|_p^p<(a+bl^2)\|u_*\|^2,
	\end{equation*}
	we infer
	$$
	\frac{S^{-2}}{S^{-2}-b}\|u_*\|^2<\frac{4-p}{ap}(a+bl^2)\|u_*\|^2,
	$$
	and using \eqref{ineql1}, it also follows
	$$
	\frac{S^{-2}}{S^{-2}-b}<\frac{4-p}{p}\left(1+b\frac{p}{p-2}\frac{1}{S^{-2}-b}\right)=\frac{4-p}{p}+b\frac{4-p}{p-2}\frac{1}{S^{-2}-b},
	$$
	from which
	$$
	b > \frac{(p-2)^2}{4-p}S^{-2},
	$$
	against the choice of $b$. So, it must be $A=\emptyset$. Consequently, by \eqref{conv} and \eqref{wsconv}, we obtain
	$$
	\|u_n\|_4\rightarrow  \|u_*\|_4 \quad \text{as } n\to\infty,
	$$
	which, together with
	$$
	u_n\rightarrow u_*  \quad\text{weakly in } L^4(\Omega) \quad \text{as } n\to\infty,
	$$
	implies
	$$
	u_n\rightarrow u_* \quad\text{strongly in } L^4(\Omega) \quad \text{as } n\to\infty.
	$$
	As a consequence, since $u_n\in \mathcal{N}_n$, for each $n\in \N$, it follows
	$$
	al^2+bl^4=\lambda \|u_*\|_4^4+\lambda\|u_*\|^p,
	$$
	which, together with \eqref{ustar}, gives
	$$
	\|u_*\|^2=l^2=\lim_{n\rightarrow +\infty}\|u_n\|^2.
	$$
	Since, $u_n\rightarrow u_*$ weakly in $\W$, we get that $u_n\rightarrow u_*$ strongly in $\W$. Observe also that, by the Strong Maximum Principle, $u_*$ is also positive in $\Omega$.  Finally, being $l^2=\|u_*\|^2$, by (\ref{euleq}) we infer that $u_*$ is a solution to $(P_\lambda)$.
\end{proof}

\medskip

The conclusion of Lemma \ref{minim3} can be also achieved for any $b\in (0,S^{-2})$ provided that $\lambda$ is small enough. Indeed, setting
\begin{equation*}
	\Lambda_p(a,b):=\frac{2a^{\frac{4-p}{2}}}{c_p^p(4-p)} \left(\frac{p-2}{p}\right)^{\frac{p-2}{2}}(S^{-2}-b)^{\frac{p-2}{2}},
\end{equation*}
we have the following result:

\begin{lemma}\label{minim4bis}
	Assume that $\lambda\leq\Lambda_p(a,b)$. Then, if $n\in \N$ is large enough and $u_n\in \mathcal{A}_n\cap \mathcal{N}_n$ is as in Lemma \ref{minim}, one has
	\begin{equation}\label{strictineq1}
		2a\|u_n\|^2+b(4-\sigma_n)\|u_n\|^{4-\sigma_n}<(4-\sigma_n)\|u_n\|^{4-\sigma_n}_{4-\sigma_n}+\lambda p \|u_n\|_p^p
	\end{equation}
\end{lemma}

\begin{proof}
\;	Let $n\in \N$ and $u_n\in \mathcal{A}_n\cap\mathcal{N}_n$ be as in Lemma \ref{minim}. One has $f(\Psi_{\lambda,n}(u))=\|u_n\|$. Then, since
	$u_n\in \mathcal{A}_n$, one has
	$$
	2a \|u_n\|^2+b(4-\sigma_n)\|u_n\|^{4-\sigma_n}\leq (4-\sigma_n)\|u_n\|^{4-\sigma_n}_{4-\sigma_n}+\lambda p \|u_n\|_p^p.
	$$
	Arguing by contradiction, we may assume (without loss of generality) that the above holds as an equality for all $n\in \N$. Then, $u_n$ satisfies
	\begin{equation*}
		\left\{\begin{array}{l}
			a\|u_n\|^2+b\|u_n\|^{4-\sigma_n}=\|u_n\|_{4-\sigma_n}^{4-\sigma_n}+\lambda  \|u_n\|_p^p\\[3mm]
			2a\|u_n\|^2+b(4-\sigma_n)\|u_n\|^{4-\sigma_n}=(4-\sigma_n)\|u_n\|_{4-\sigma_n}^{4-\sigma_n}+\lambda p\|u_n\|_p^p\\[3mm]
			\displaystyle\frac{a(p-2)(2-\sigma_n)}{2p(4-\sigma_n)}\|u_n\|^2=I_{\lambda,n}(u_n)<\displaystyle\frac{a^2}{4(S^{-2}-b)}-\gamma,
		\end{array}\right.
	\end{equation*}
	from which
	\begin{equation*}
		\left\{\begin{array}{l}
			\displaystyle{\lambda = \frac{a(2-\sigma_n)}{4-p-\sigma_n}\frac{\|u_n\|^2}{\|u_n\|_p^p}}\\[4mm]
			\displaystyle{\|u_n\|^2 <  \frac{2p(4-\sigma_n)}{a(p-2)(2-\sigma_n)}    \left(\frac{a^2}{4(S^{-2}-b)}-\gamma\right)}.
		\end{array}\right.
	\end{equation*}
	Therefore,
	\begin{align*}
		\lambda & = \frac{a(2-\sigma_n)}{4-p-\sigma_n}\left(\frac{\|u_n\|}{\|u_n\|_p}\right)^p\|u_n\|^{2-p}\\
		&\geq
		\frac{a(2-\sigma_n)}{c_p^p(4-p-\sigma_n)}\left[\frac{2p(4-\sigma_n)}{a(p_2)(2-\sigma_n)}\left(\frac{a^2}{4(S^{-2}-b)}-\gamma\right)\right]^{\frac{2-p}{2}}
	\end{align*}
	Letting $n\rightarrow +\infty$, we finally get
	\begin{align*}
		\lambda & \geq \frac{2a}{c_p^p(4-p)}
		\left[\frac{4p}{a(p-2)}\left(\frac{a^2}{4(S^{-2}-b)}-\gamma\right)\right]^{\frac{2-p}{2}}\\
		&>\frac{2a^{\frac{4-p}{2}}}{c_p^p(4-p)}\left(\frac{p-2}{p}\right)^{\frac{p-2}{2}}(S^{-2}-b)^{\frac{p-2}{2}}\\
		&=\Lambda_p(a,b),
	\end{align*}
	against the assumption.
\end{proof}

Arguing as above, it is straightforward to verify that the same conclusion of Lemma \ref{minim4} can be obtained again for the full range $b\in (0,S^{-2})$, provided that $\lambda\leq \Lambda_p(a,b)$.

\begin{theorem}\label{minim5bis}
	Let $\lambda\leq \Lambda_p(a,b)$ and, for each $n\in \N$, let $u_n\in \mathcal{A}_n\cap \mathcal{N}_n$ be as in Lemma \ref{minim2}.  Then, $\{u_n\}$
	strongly converges in $\W$ to a solution $u^*$ of problem \eqref{problem}.
\end{theorem}

\begin{proof}
\;	We use exactly the same arguments of the proof of Theorem \ref{minim5}, but exploiting Lemma \eqref{minim4bis} instead of Lemma \eqref{minim4}. Doing so, assuming that $\mu_k\neq 0$ for some $k\in \N$, we infer that, for $\lambda\leq \Lambda_p(a,b)$, inequalities \eqref{ineql1},
	\eqref{inequstar} and \eqref{inequstar1} hold. In particular, by these inequalities we get
	\begin{equation*}
		\|u_*\|^2\leq \frac{2aS^2}{p-2} \quad\text{and}\quad \lambda \geq \frac{apS^{-2}}{c_p^p(4-p)(S^{-2}-b)}\cdot \frac{1}{\|u_*\|^{p-2}},
	\end{equation*}
	from which
	\begin{align*}
		\lambda & \geq \frac{a^{\frac{4-p}{2}}}{c_p^p}\cdot \frac{p}{4-p}\left(\frac{p-2}{2}\right)^{\frac{p-2}{2}}\cdot \frac{(S^{-2})^{\frac{p}{2}}}{S^{-2}-b}\\
		& >\frac{a^{\frac{4-p}{2}}}{c_p^p}\cdot \frac{p}{4-p}\left(\frac{p-2}{2}\right)^{\frac{p-2}{2}}\cdot
		(S^{-2}-b)^{\frac{p-2}{2}}\\
		&= \left( \frac{p}{2}\right)^\frac{p}{2}\Lambda_p(a,b)>\Lambda_p(a,b),
	\end{align*}
	a contradiction. Hence, it must be $\mu_k=0$ for all $k\in A$, and by this, as in the proof of Lemma \ref{minim5}, the conclusion follows.
\end{proof}

\section*{Acknowledgment}
The authors are members of the Gruppo Nazionale per l'Analisi Matematica, la Probabilità e le loro Applicazioni (GNAMPA) of the Istituto Nazionale di Alta Matematica (INdAM).



\begin{thebibliography}{99}
	
\bibitem{bn} H. Brezis and L. Nirenberg. Positive solutions of nonlinear elliptic equations involving critical Sobolev exponents. {\it Commun. Pure and Appl. Math}. {\bf 36} (1983), 437--477.

\bibitem{alvcorma2005positive} C.O. Alves, F.J.S.A. Corr\^{e}a and T.F. Ma. Positive solutions for a quasilinear elliptic equation of Kirchhoff type. {\it Comput. Math. Appl}. {\bf 49} (2005), 85--93.

\bibitem{alvcorfig2010on} C.O. Alves, F.J.S.A. Corr\^{e}a and G.M. Figueiredo. On a class of nonlocal elliptic problems with critical growth. {\it Differ. Equ. Appl}. {\bf 2} (2010), 409--417.

\bibitem{camvil2014a} F. Cammaroto and L. Vilasi. A critical Kirchhoff-type problem involving the $p$\&$q$-Laplacian. {\it Math. Nachr}. {\bf 287}(2-3) (2014), 184--193.

\bibitem{fan2015multiple} H. Fan. Multiple positive solutions for a class of Kirchhoff type problems involving critical Sobolev exponents. {\it J. Math. Anal. Appl}. {\bf 431} (2015),  150–-168.

\bibitem{farsil2021on} F. Faraci and K. Silva. On the Brezis-Nirenberg problem for a Kirchhoff type equation in high dimension. {\it Calc. Var}. (2021), 60:22.

\bibitem{f} G.M. Figueiredo. Existence of a positive solution for a Kirchhoff problem type with critical growth via truncation argument. {\it J. Math. Anal. Appl}. {\bf 401} (2013), 706--713.

\bibitem{kir1883vorlesungen} G.R. Kirchhoff. Vorlesungen \"{u}ber Mathematische Physik: Mechanik. Teubner, Leipzig, 1883.

\bibitem{mamor2024asymptotic} S. Ma and V. Moroz. Asymptotic profiles for a nonlinear Kirchhoff equation with combined powers nonlinearity. {\it Nonlinear Anal. TMA} {\bf 239} (2024), 113423.

\bibitem{nai2014the} D. Naimen. The critical problem of Kirchhoff type elliptic equations in dimension four. {\it J. Differential Equations} {\bf 257} (2014), 1168--1193.

\bibitem{nai2014positive} D. Naimen. Positive solutions of Kirchhoff type elliptic equations involving a critical Sobolev exponent. {\it NoDEA Nonlinear Differ. Equ. Appl}. {\bf 21}(6) (2014), 885--914.

\bibitem{naishi2020existence} D. Naimen and M. Shibata. Existence and multiplicity of positive solutions of a critical Kirchhoff type elliptic problem in dimension four. {\it Differential Integral Equations} {\bf 33}(5-6) (2020), 223--246.

\bibitem{pucrad2019progress} P. Pucci and V.D. R\u{a}dulescu. Progress in nonlinear Kirchhoff problems. {\it Nonlinear Anal}. {\bf 186} (2019), 1--5.

\end{thebibliography}
\end{document}